\documentclass[preprint,12pt]{elsarticle}

\usepackage{amsmath, amssymb, amsthm}
\usepackage{mathrsfs}
\usepackage{dsfont}
\usepackage{graphicx}
\usepackage{xcolor}
\usepackage{tikz}
\usetikzlibrary{positioning, decorations.pathreplacing, calc, shapes.geometric}
\usepackage[margin=2.5cm]{geometry}
\usepackage{hyperref}
\newcommand{\paramsep}{\mathpunct{\text{;}}}

\newtheorem{theorem}{Theorem}[section]
\newtheorem{proposition}[theorem]{Proposition}

\newtheorem{lemma}[theorem]{Lemma}
\newtheorem{corollary}[theorem]{Corollary}
\newtheorem{remark}[theorem]{Remark}

\newtheorem{example}[theorem]{Example}

\numberwithin{equation}{section}
\numberwithin{figure}{section}

\DeclareMathOperator{\rank}{rank}

\begin{document}

\begin{frontmatter}

\title{Trees with exactly three main eigenvalues}

\author[1]{Hangxi Cha}
\ead{2433949@tongji.edu.cn}
\author[1]{Haiying Shan\corref{cor1}}
\ead{shan_haiying@tongji.edu.cn}
\affiliation[1]{organization={School of Mathematical Sciences, Tongji University},
            addressline={},
            city={Shanghai},
            postcode={200092},
            state={},
            country={P.R. China}}
\cortext[cor1]{Corresponding author}
\nonumnote{This work is partially supported by the National Natural Science Foundation of China (No. 12271182).}

\begin{abstract}
An eigenvalue of a graph is called main if its eigenspace is not orthogonal to the all-ones vector. Introduced by Cvetkovi\'{c} in the early 1970s and systematically studied by Rowlinson and others, graphs with exactly one or two main eigenvalues are now well understood. However, the classification of graphs with precisely three main eigenvalues remains a challenging open problem in spectral graph theory.
This paper provides a complete classification of all trees of diameter 5 with exactly three main eigenvalues. Using equitable partitions, the spectral condition reduces to the unique solvability of linear systems over the rationals, leading to Diophantine equations involving branch lengths and pendant counts. We prove that every such tree is isomorphic either to a symmetric tree \(T_r(a)\) or to a member of a parametric family \(\mathcal{T}\) determined by arithmetic divisibility conditions. We also construct an infinite family of such trees with unbounded diameter.
\end{abstract}


\end{frontmatter}

\section{Introduction}

All graphs considered in this paper are finite, simple, and undirected. Let $G=(V(G), E(G))$ be a graph with vertex set $V(G)=\{v_1, \dots, v_n\}$ and adjacency matrix $A$. The eigenvalues of $A$, denoted by $\mu_1 \ge \mu_2 \ge \dots \ge \mu_n$, form the spectrum of $G$. Let $\tau_1, \tau_2, \dots, \tau_\nu$ be the distinct eigenvalues of $A$. The spectral decomposition of $A$ is given by $A = \sum_{i=1}^\nu \tau_i P_i$, where $P_i$ represents the orthogonal projection matrix onto the eigenspace $\mathcal{E}(\tau_i)$. Let $\mathbf{j}$ denote the all-ones column vector of dimension $n$. The \emph{main angle} $\beta_i$ of $G$ associated with $\tau_i$ is defined as $\beta_i = \frac{1}{\sqrt{n}} \|P_i \mathbf{j}\|$. An eigenvalue $\tau_i$ is called a \emph{main eigenvalue} if $\beta_i \neq 0$, which means the eigenspace $\mathcal{E}(\tau_i)$ is not orthogonal to $\mathbf{j}$. The set of distinct main eigenvalues, denoted by $\operatorname{Spec}_{\rm M}(G)=\{\lambda_1, \dots, \lambda_k\}$, is called the \emph{main spectrum} of $G$, where $k$ denotes the number of main eigenvalues. Accordingly, the monic polynomial $\varphi_{\rm M}(x) = \prod_{i=1}^k (x-\lambda_i)$ is defined as the \emph{main characteristic polynomial} of $G$.

Main eigenvalues play a fundamental role in the enumeration of walks. For any vertex $v \in V(G)$ and non-negative integer $\ell$, let $d_\ell(v)$ denote the number of walks of length $\ell$ in $G$ starting at vertex $v$. It is a standard result in algebraic graph theory that the vector of these walk counts is given by $A^\ell \mathbf{j}$, that is, $d_\ell(v)$ corresponds to the entry of the vector $A^\ell \mathbf{j}$ indexed by $v$. In particular, $d_0(v) = 1$ for all $v$, and $d_1(v)$ is simply the degree of $v$, often denoted by $d(v)$.
Let $N_\ell$ denote the total number of walks of length $\ell$ in $G$. It is well known that $N_\ell = \sum_{v \in V(G)} d_\ell(v) = \mathbf{j}^\top A^\ell \mathbf{j}$. By substituting the spectral decomposition into this expression and noting that non-main eigenvalues do not contribute to the sum, we obtain the formula
\[
N_\ell = n \sum_{i=1}^k \beta_i^2 \lambda_i^\ell.
\]
This relationship highlights that the walk enumeration is exclusively determined by the main eigenvalues and their corresponding main angles.

A classic problem in spectral graph theory, posed by Cvetkovi\'{c} \cite{CvRoSi}, is to characterize graphs with exactly $k$ main eigenvalues. The case $k=1$ is well understood: a graph has exactly one main eigenvalue if and only if it is regular. For $k \ge 2$, the problem becomes significantly more challenging. Hagos \cite{Hag} proved that $G$ has exactly $k$ main eigenvalues if and only if the rank of the walk matrix $W = [\mathbf{j}, A\mathbf{j}, \dots, A^{n-1}\mathbf{j}]$ is $k$. Note that the columns of $W$ are precisely the walk count vectors $(d_0(v))_{v \in V}, (d_1(v))_{v \in V}, \dots, (d_{n-1}(v))_{v \in V}$. Based on this criterion, classifications have been achieved for graphs with $k=2$, such as trees \cite{HouZhou2005}, unicyclic graphs \cite{HoTi}, bicyclic graphs \cite{HuLiZh}, and tricyclic graphs \cite{FaLuGa}. In contrast, the classification of graphs with exactly three main eigenvalues ($k=3$) remains largely open. Rowlinson \cite{Row} investigated the specific case where $\varphi_{\rm M}(x) = x(x^2 - \mu^2)$, characterizing such graphs as pseudo-semi-regular bipartite graphs. Recently, Fran\c{c}a, Brondani, and Jaume \cite{FrBrJa} made significant progress by characterizing all trees of diameter 4 with exactly three main eigenvalues via equitable partitions.

In this paper, we extend this line of research to trees of diameter 5. We provide a complete classification of all trees of diameter 5 with exactly three main eigenvalues. Our main result shows that the existence of such trees is completely characterized by the solvability of certain systems of Diophantine equations involving their structural parameters. We also give an infinite family of trees with exactly three main eigenvalues and unbounded diameter.

\section{Preliminaries}

A partition $\pi = \{V_1, V_2, \dots, V_m\}$ of the vertex set $V(G)$ is called an equitable partition if for any pair of indices $i, j \in \{1, \dots, m\}$, every vertex $v \in V_i$ has exactly the same number of neighbors in $V_j$. This constant number is denoted by $b_{ij}$.
Given an equitable partition $\pi$, the divisor (or quotient graph) of $G$ is a directed multigraph $D_\pi$ with vertices $V_1, \dots, V_m$ and $b_{ij}$ arcs from $V_i$ to $V_j$. The $m \times m$ matrix $B_\pi = (b_{ij})$ is called the divisor matrix (or quotient matrix), and we denote its characteristic polynomial by $\varphi_{B_\pi}(x) = \det(xI - B_\pi)$. Let $C$ be the $n \times m$ characteristic matrix of $\pi$, where the $j$-th column is the characteristic vector of the cell $V_j$. Note that $C^\top C = \text{diag}(|V_1|, \dots, |V_m|)$.

\begin{proposition}[{\cite[Proposition 3.9.3]{CvRoSi}} ]
Let $G$ be a graph with adjacency matrix $A$. If $\pi$ is an equitable partition with divisor matrix $B_\pi$ and characteristic matrix $C$, then the column space of $C$ is $A$-invariant, and the following relations hold:
\[
\begin{aligned}
AC &= CB_\pi,
  B_\pi = (C^\top C)^{-1} C^\top A C.
\end{aligned}
\]
\end{proposition}

This algebraic relationship allows us to relate the spectrum of the divisor to the spectrum of the original graph.

\begin{theorem}[{\cite[Theorem 3.9.5]{CvRoSi}} ]
The characteristic polynomial $\varphi_{B_\pi}(x)$ of any divisor of a graph $G$ divides the characteristic polynomial of $G$.
\end{theorem}

\begin{remark}
As a consequence of the relation $AC=CB_\pi$, for any polynomial $f(x) \in \mathbb{R}[x]$, we have 
\[
    f(A)C = Cf(B_\pi).
\]
Furthermore, a vector $v \in \mathbb{R}^m$ is an eigenvector of $B_\pi$ if and only if $Cv$ is an eigenvector of $A$. This implies that the lifted eigenvectors $Cv$ of $A$ associated with the eigenvalues of the divisor are constant on the cells of $\pi$.
\end{remark}

\begin{lemma}[{\cite[Theorem 2.5]{Ter}} ] \label{le2.4}
Let $f(x) \in \mathbb{R}[x]$. Then $f(A)\mathbf{j} = \mathbf{0}$ if and only if $\varphi_M(x)$ divides $f(x)$.
\end{lemma}

By combining the concepts of equitable partitions and main eigenvalues, we arrive at a fundamental result linking the divisor's spectrum to the main eigenvalues of the graph.

\begin{theorem}[{\cite[Theorem 3.9.9]{CvRoSi}} ]
The characteristic polynomial $\varphi_{B_\pi}(x)$ of a divisor of a graph $G$ is divisible by $\varphi_M(x)$.
\end{theorem}

Since $\varphi_M(x)$ divides $\varphi_{B_\pi}(x)$, it follows that the set of main eigenvalues is a subset of the eigenvalues of the divisor. Consequently, the number of main eigenvalues of $G$ is bounded above by the degree of $\varphi_{B_\pi}(x)$, which is equal to the order of the divisor (i.e., the number of cells $m$ in the equitable partition $\pi$).

\begin{theorem}[{\cite[Theorem 2.1]{Hag}} ]\label{thm:walkrank}
The rank of the walk-matrix $W(G)$ is equal to the number of main eigenvalues.
\end{theorem}

\begin{lemma}[{\cite[Proposition 3]{HoTi}} ]\label{le2.7}
Let $\lambda_1, \lambda_2, \ldots, \lambda_k$ be all the main eigenvalues of $G$. Then $\varphi_M(x) = (x - \lambda_1)(x - \lambda_2)\cdots(x - \lambda_k)$ is an integral polynomial.
\end{lemma}

\begin{proposition}[{\cite[Lemma 2.4]{HuHuLu}} ]\label{prop:quotient-walk-rank}
Let $\pi = \{V_1, V_2, \dots, V_r\}$ be an equitable partition of $G$, and let $B_\pi$ be the quotient matrix of $G$ with respect to $\pi$. Then the number of main eigenvalues of $G$ is equal to $\operatorname{rank}(W(B_\pi))$.
\end{proposition}

\begin{theorem}\label{t1}
Let $G$ be a non-regular connected graph. Then $G$ has exactly $k$ main eigenvalues if and only if there exist unique integers $c_0, c_1, \dots, c_{k-1}$ such that
\begin{equation}\label{eq:scalar_relation}
d_k(v) = c_0 d_{k-1}(v) + c_1 d_{k-2}(v) + \dots +c_{k-2}d(v)+c_{k-1}
\end{equation}
holds for every vertex $v \in V$.
\end{theorem}

\begin{proof}
Condition \eqref{eq:scalar_relation} holds for all $v$ if and only if the following vector equation holds:
\[
A^k \mathbf{j} = c_0 A^{k-1} \mathbf{j} + c_1 A^{k-2} \mathbf{j} + \dots + c_{k-1} \mathbf{j}.
\]

Suppose $G$ has exactly $k$ main eigenvalues $\lambda_1, \dots, \lambda_k$. By Theorem~\ref{thm:walkrank}, the rank of the walk matrix $W = [\mathbf{j}, A\mathbf{j}, \dots, A^{n-1}\mathbf{j}]$ is $k$. Consequently, the set of vectors $\{\mathbf{j}, A\mathbf{j}, \dots, A^{k-1}\mathbf{j}\}$ is linearly independent and forms a basis for the $A$-invariant subspace generated by $\mathbf{j}$. This implies that $A^k \mathbf{j}$ can be uniquely expressed as a linear combination of these vectors.

Let $\varphi_M(x) = \prod_{i=1}^k (x - \lambda_i) = x^k - \sum_{i=0}^{k-1} c_i x^{k-1-i}$. By Lemma \ref{le2.7}, $\varphi_M(x)$ is a monic integral polynomial, and $c_i \in \mathbb{Z}$. By Lemma \ref{le2.4}, $\varphi_M(A)\mathbf{j} = \mathbf{0}$, which yields $A^k \mathbf{j} = \sum_{i=0}^{k-1} c_i A^{k-1-i} \mathbf{j}$. Thus, integer coefficients exist. The uniqueness of these integers is guaranteed by the linear independence of the basis vectors.

Conversely, suppose that there exist unique integers $c_0, c_1, \dots, c_{k-1}$ such that \eqref{eq:scalar_relation} holds for every vertex $v$. This condition is equivalent to the vector equation
$
A^k \mathbf{j} = \sum_{i=0}^{k-1} c_i A^{k-1-i} \mathbf{j}.
$
The uniqueness of the coefficients $c_i$ implies that the set of vectors $\{\mathbf{j}, A\mathbf{j}, \dots, A^{k-1}\mathbf{j}\}$ is linearly independent. Consequently, the rank of the walk matrix $W$ is exactly $k$. By Theorem~\ref{thm:walkrank}, it follows that $G$ has exactly $k$ main eigenvalues.
\end{proof}

\begin{corollary}\label{1c}
Let $G$ be a non-regular connected graph with equitable partition $\pi=\{V_1,V_2,\ldots,V_m\}$. Then $G$ has exactly $k$ main eigenvalues if and only if there exist unique integers $c_0, c_1, \dots, c_{k-1}$ such that
$$
d_k(v_i) = c_0 d_{k-1}(v_i) + c_1 d_{k-2}(v_i) + \dots +c_{k-2}d(v_i)+c_{k-1}.
$$
for every vertex $v_i \in V_i$.
\end{corollary}

\section{Trees of diameter 5 with 3 main eigenvalues}

A tree of diameter 5 is obtained by joining two adjacent vertices $u_1$ and $u_2$ to the centers of $r, s \geq 1$ stars, denoted by $K_{1,a_1}, K_{1,a_2}, \ldots, K_{1,a_r}$ and $K_{1,b_1}, K_{1,b_2}, \ldots, K_{1,b_s}$ respectively, and to sets of $p, q \geq 0$ pendant vertices attached directly to $u_1$ and $u_2$, respectively.

Such a tree is denoted by $T^{p,q}(a_1, \ldots, a_r \paramsep b_1, \ldots, b_s)$, where $1 \leq a_1 \leq a_2 \leq \cdots \leq a_r$ and $1 \leq b_1 \leq b_2 \leq \cdots \leq b_s$.
When $a_i = a$ for $1 \leq i \leq r$, and $b_j = b$ for $1 \leq j \leq s$, the tree is denoted by $T_{r,s}^{p,q}(a,b)$.
If $p=q=0$, we simply write $T(a_1, \ldots, a_r \paramsep b_1, \ldots, b_s)$ .
Furthermore, if $p=q$, $r=s$, and $a=b$, we denote the symmetric tree by $T_{r}^{p}(a)$.

We define the partition 
\[
\pi(T^{p,q}(a_1, \ldots, a_r \paramsep b_1, \ldots, b_s)) = \{V_1, V_2, \ldots, V_{2r+2s+4}\}
\]
of the vertex set of $T^{p,q}(a_1, \ldots, a_r \paramsep b_1, \ldots, b_s)$ as follows. Let the vertices be labeled such that:

    For $1 \leq k \leq r$, $V_k$ is the set of pendent vertices of the $k$-th star attached to $u_1$.
     For $1 \leq k \leq s$, $V_{r+k}$ is the set of pendent vertices of the $k$-th star attached to $u_2$.
     For $1 \leq k \leq r$, $V_{r+s+k}$ equals $w_k,$ the center of the $k$-th star attached to $u_1$.
     For $1 \leq k \leq s$, $V_{2r+s+k}$ equals $z_k,$ the center of the $k$-th star attached to $u_2$.
   $V_{2r+2s+1} = \{u_1\}$ , $V_{2r+2s+2} = \{u_2\}$.
    $V_{2r+2s+3}$ is the set of the $p$ pendant vertices adjacent to $u_1$.  $V_{2r+2s+4}$ is the set of the $q$ pendant vertices adjacent to $u_2$.

It is straightforward to verify that this partition is an equitable partition of the tree.
Figure~\ref{fig:diam5-partition} displays the $\pi(T^{p,q}(a_1, \ldots, a_r \paramsep b_1, \ldots, b_s))$.

\def\myScale{0.85}      
\def\xRoot{2.2}         
\def\yTop{1}           
\def\yMid{-1.5}         
\def\yBot{-2.5}         
\def\xTopSpread{0.6}    
\def\xMidOuter{3.4}     
\def\xMidInner{1.0}     
\def\leafSpread{0.7}    
\begin{figure}
    \centering
    \begin{tikzpicture}[
    scale=\myScale, 
    every node/.style={scale=\myScale},
    vertex/.style={circle, fill=black, inner sep=0pt, minimum size=4pt},
    edge/.style={draw=black, thick}, 
    mylabel/.style={font=\small, text=black},
    brace/.style={decorate, decoration={brace, amplitude=4pt}, thick},
    brace_mirror/.style={decorate, decoration={brace, amplitude=4pt, mirror}, thick},
    brace_label/.style={midway, font=\scriptsize, align=center, inner sep=4pt}
]

    \coordinate (v1) at (-\xRoot, 0);
    \coordinate (v2) at (\xRoot, 0);

    \node[vertex] (N_v1) at (v1) {};
    \node[vertex] (N_v2) at (v2) {};
    \draw[edge] (N_v1) -- (N_v2);

    \node[mylabel, left=6pt] at (N_v1) {$u_1$};
    \node[mylabel, right=6pt] at (N_v2) {$u_2$};

    \coordinate (cL) at ($(v1) + (-\xTopSpread, \yTop)$);
    \coordinate (cR) at ($(v1) + (\xTopSpread, \yTop)$);
    \node[vertex] (N_cL) at (cL) {};
    \node[vertex] (N_cR) at (cR) {};
    \node at ($(cL)!0.5!(cR)$) {$\cdots$};
    \foreach \x in {N_cL, N_cR} \draw[edge] (N_v1) -- (\x);
    
    \draw[brace] (N_cL.north west) -- (N_cR.north east)
        node[brace_label, above=3pt] {$V_{2r+2s+3}$ };

    \coordinate (dL) at ($(v2) + (-\xTopSpread, \yTop)$);
    \coordinate (dR) at ($(v2) + (\xTopSpread, \yTop)$);
    \node[vertex] (N_dL) at (dL) {};
    \node[vertex] (N_dR) at (dR) {};
    \node at ($(dL)!0.5!(dR)$) {$\cdots$};
    \foreach \x in {N_dL, N_dR} \draw[edge] (N_v2) -- (\x);
    
    \draw[brace] (N_dL.north west) -- (N_dR.north east)
        node[brace_label, above=3pt] {$V_{2r+2s+4}$ };

    
    \coordinate (L_outer) at (-\xMidOuter, \yMid);
    \coordinate (L_inner) at (-\xMidInner, \yMid);

    \node[vertex] (N_L_out) at (L_outer) {};
    \node[vertex] (N_L_in) at (L_inner) {};
    \node at ($(L_outer)!0.5!(L_inner)$) {$\cdots$};

    \draw[edge] (N_v1) -- (N_L_out);
    \draw[edge] (N_v1) -- (N_L_in);

    \node[mylabel, above left] at (N_L_out) {$w_1$};
    \node[mylabel, above right] at (N_L_in) {$w_r$};

    \coordinate (R_inner) at (\xMidInner, \yMid);
    \coordinate (R_outer) at (\xMidOuter, \yMid);

    \node[vertex] (N_R_in) at (R_inner) {};
    \node[vertex] (N_R_out) at (R_outer) {};
    \node at ($(R_inner)!0.5!(R_outer)$) {$\cdots$};

    \draw[edge] (N_v2) -- (N_R_in);
    \draw[edge] (N_v2) -- (N_R_out);

    \node[mylabel, above left] at (N_R_in) {$z_1$};
    \node[mylabel, above right] at (N_R_out) {$z_s$};

    \newcommand{\drawLeaves}[2]{
        \coordinate (l_left) at ($(#1) + (-\leafSpread, \yBot-\yMid)$);
        \coordinate (l_right) at ($(#1) + (\leafSpread, \yBot-\yMid)$);
        
        \node[vertex] (n_l) at (l_left) {};
        \node[vertex] (n_r) at (l_right) {};
        
        \draw[edge] (#1) -- (n_l);
        \draw[edge] (#1) -- (n_r);
        
        \node at ($(l_left)!0.5!(l_right)$) {$\cdots$};
        
        \draw[brace_mirror, shorten <=0pt, shorten >=0pt] 
            (n_l.south west) -- (n_r.south east)
            node[brace_label, below=2pt] {#2};
    }

    \drawLeaves{N_L_out}{$V_{1}$ }
    \drawLeaves{N_L_in}{$V_{r}$ }
    \drawLeaves{N_R_in}{$V_{r+1}$ }
    \drawLeaves{N_R_out}{$V_{r+s}$ }

       
\end{tikzpicture}
    \caption[Labeled diameter-5 tree]{ Labeled \texorpdfstring{$T^{p,q}(a_1,\ldots,a_r \paramsep b_1,\ldots,b_s)$}{T^(p,q)(a1,...,ar; b1,...,bs)}.}
    \label{fig:diam5-partition}
\end{figure}

We define the following family of symmetric trees of diameter 5:

\[
\mathcal{T} = \left\{ T_{\delta(\eta-\kappa)}^{\kappa^2-\delta(\eta-\kappa)} \left(\frac{\kappa^3}{\eta}\right) \;\middle|\; \eta \mid \kappa^3, \; 1 \le \delta(\eta-\kappa) \le \kappa^2-1 \right\},
\]
where $\kappa \ge 2$, $\eta \ge 1$, and $\delta \in \{1, -1\}$.

In this section, Theorem~\ref{t2} , \ref{t3} , \ref{t4} completely characterize trees of diameter 5 with exactly three main eigenvalues. The following lemmas are key parts of their proofs.

\begin{lemma}\label{le1}
  Let $\kappa \ge 2$ be an integer and $\delta \in \{1,-1\}$. Consider
  \[
  \begin{cases}
  r + p = \kappa^2,\\[2pt]
  a\bigl(\kappa+\delta r\bigr) = \kappa^3,
  \end{cases}
  \]
  in positive integers $r,a,p$. Then there is a solution if and only if there exists a positive divisor $\eta$ of $\kappa^3$ such that
  \[
  1 \le \delta(\eta-\kappa) \le \kappa^2 - 1,
  \]
  in which case
  \[
  r = \delta(\eta-\kappa),\qquad a = \dfrac{\kappa^3}{\eta},\qquad p = \kappa^2 - \delta(\eta-\kappa).
  \]
  \end{lemma}
  
  \begin{proof}
  Since $a(\kappa+\delta r)=\kappa^3$, set $\eta := \kappa+\delta r$ (so $\eta$ is a positive divisor of $\kappa^3$). Then $a = \kappa^3/\eta$ and $r = \delta(\eta-\kappa)$. From $r+p=\kappa^2$ we find $p = \kappa^2 - \delta(\eta-\kappa)$. The constraints $r, p > 0$ translate to $1 \leq \delta(\eta-\kappa) \leq \kappa^2-1$.

  Conversely, if $\eta \mid \kappa^3$ and $1 \leq \delta(\eta-\kappa) \leq \kappa^2-1$, set $a=\kappa^3/\eta$, $r=\delta(\eta-\kappa)$, $p=\kappa^2-\delta(\eta-\kappa)$. Then $a>0$, $r>0$, $p>0$ and both original equations clearly hold.
  \end{proof}

To address the solvability of the polynomial systems appearing in the following lemmas, we utilize the theory of Gr\"obner bases. Let $I$ be an ideal in a polynomial ring $S$, and let $<$ be a monomial order on $S$. A sequence $g_1, \dots, g_m$ of elements in $I$ is called a \emph{Gr\"obner basis} of $I$ with respect to $<$ if the initial ideal $\mathrm{in}_<(I)$ is generated by the leading terms of the sequence, that is, $\mathrm{in}_<(I) = \langle \mathrm{in}_<(g_1), \dots, \mathrm{in}_<(g_m) \rangle$ \cite[Definition 2.5]{EneHerzog2012}.

\begin{remark}\label{rem:gb_generates}
A fundamental consequence of this definition is that if $g_1, \dots, g_m$ is a Gr\"obner basis of $I$ with respect to a monomial order $<$, then the sequence $g_1, \dots, g_m$ generates the ideal $I$ \cite[Theorem 2.8]{EneHerzog2012}. This ensures that the system of equations defined by the Gr\"obner basis shares the same solution set as the original system.
\end{remark}

To compute the Gr\"obner basis explicitly, we employ Buchberger's Algorithm, as detailed in \cite[Chapter 2, Section 7]{CLO2015}. This constructive method iteratively transforms a given generating set into a Gr\"obner basis in a finite number of steps. The correctness of the algorithm relies on Buchberger's Criterion \cite[Theorem 2.14]{EneHerzog2012}. In the proofs that follow, we apply this algorithmic machinery to compute elimination ideals and strictly determine the consistency of the associated systems of Diophantine equations.

\begin{lemma}\label{le2}
    Let $a, b, r, s \in \mathbb{Z}_{>0}$ be positive integers such that $a \neq b$. Define the constant
    \[
        K = 1 + \frac{s-r}{b-a}.
    \]
    Define $L_i$ and $R_i$ for $i \in \{1, 2, 3\}$ as follows:
    \begin{align*}
        L_1 &= a(b^2+b-a+s+sb) - b(a^2+s+ra), & R_1 &= a(b-a+s) - br,  \\
        L_2 &= a(r^2+r+s+ra+sb-a) - r(a^2+s+ra), & R_2 &= a(s+ra+r-a) - r^2,  \\
        L_3 &= a(s^2+2s+ra+sb-a) - s(a^2+s+ra), & R_3 &= a(r+sb+s-a) - sr. 
    \end{align*}
    Then, the system of Diophantine equations $\mathcal{S}$ given by
    \[
        \mathcal{S}: \quad L_i - R_i K = 0 \quad \text{for } i=1, 2, 3
    \]
    has no solutions.
\end{lemma}

\begin{proof}
    Suppose, for the sake of contradiction, that solutions exist. Substituting $K$ into the system yields the polynomial equations:
    \[
        E_i(a,b,r,s) := (b-a)(L_i - R_i) - R_i(s-r) = 0, \quad \text{for } i=1, 2, 3.
    \]

    Let $I = \langle E_1, E_2, E_3 \rangle$ be the ideal generated by these polynomials in the ring $\mathbb{Q}[a,b,r,s]$.  To eliminate the variables \( r \) and \( s \), we compute a Gr\"obner basis of \( I \) with respect to the lexicographic monomial order \( r \succ s \succ a \succ b \) (see the Data Availability Statement for the SageMath code and computational details).  The computation reveals that there is a unique polynomial in the basis depending only on $a$ and $b$, namely
    \[ P(a,b) = a^2 b (a-b)^3 Q(a,b),
    \]
  where 
 \[
        Q(a,b) = ab \left[ (a-1)^2 + (b-1)^2 \right] + (a^2b^2 - a^2 - b^2).
 \]

  Since  $a\neq b$, the existence of a solution to $\mathcal{S}$ implies that $Q(a,b) = 0$.
    
    Since $Q(a,b)$ is symmetric, without loss of generality, assume $a > b \ge 1$.
    If $b=1$, $Q(a,1) = a^3 - 2a^2 + a - 1 = a^2(a-2) + (a-1)$, which is strictly positive for all $a \ge 2$.
    If $b \ge 2$, then $a > 2$, so $a^2b^2 - a^2 - b^2 > 0$, and hence $Q(a,b) > 0$.
    Consequently, $Q(a,b) > 0$ for all  $a, b\in \mathbb{Z}_{>0}$, a contradiction. Thus, no solutions exist.
\end{proof}

\begin{lemma}\label{le3}
    Let $a, b, r, s, p, q \in \mathbb{Z}_{>0}$ be positive integers such that $a \neq b$. Define the constant
    \[
        K = 1 + \frac{s+q-r-p}{b-a}.
    \]
    Define $L_i$ and $R_i$ for $i \in \{1, \dots, 3\}$ as follows:
    \begin{align*}
        L_1 &= ra - a + s + q, & R_1 &= r + p - a, \\
        L_2 &= sb + s + q - a, & R_2 &= s + q - a, \\
        L_3 &= a(b^2+b+sb+s+q-a) - b(a^2+ra+s+q), & R_3 &= a(b+s+q-a) - b(r+p), \\           
    \end{align*}
    Then, the system of Diophantine equations $\mathcal{S}$ given by
    \[
        \mathcal{S}: \quad L_i - R_i K = 0 \quad \text{for } i=1, \dots, 3
    \]
    has no solutions.
\end{lemma}

\begin{proof}
    Suppose, for the sake of contradiction, that solutions exist. Substituting $K$ into the system yields the polynomial equations:
    \[
        E_i(a,b,r,s,p,q) := (b-a)(L_i - R_i) - R_i(s+q-r-p) = 0, \quad \text{for } i=1, \dots, 3.
    \]
    Let $I = \langle E_1, \dots, E_3 \rangle$ be the ideal generated by these polynomials in the ring $\mathbb{Q}[a,b,r,s,p,q]$. Consider the elimination ideal $I \cap \mathbb{Q}[a,b]$. Using the Buchberger algorithm with the lexicographic ordering $r \succ s \succ p \succ q \succ a \succ b$, we compute the Gr\"obner basis of $I$ using SageMath (see \hyperref[sec:data-availability]{Data Availability Statement}). The computation reveals that there is a unique polynomial in the basis depending only on $a$ and $b$, namely
    \[
        P(a,b) = a b (a-b)^2.
    \]
    Since $a \neq b$, the existence of a solution to $\mathcal{S}$ implies that $P(a,b) = 0$.

    However,  $P(a,b) > 0$ for all $a, b \in \mathbb{Z}_{>0}$ with $a \neq b$, a contradiction. Thus, no solutions exist.
\end{proof}

\begin{lemma}\label{lem:no-solution-3eq}
    Let $a,b,r,s,p \in \mathbb{Z}_{>0}$ be positive integers with $a \neq b$. Set $u := s-a$ and $v := r+p-b$. Define:
    \begin{equation*}
        K = 1 + \frac{s-r-p}{b-a}.
    \end{equation*}
    Assume that $K$ is an integer. Let $L_i, R_i$ for $i \in \{1,2,3\}$ be defined by
    \begin{align*}
        L_1 &= ar + s - a, & R_1 &= r + p - a,\\
        L_2 &= s(ab+a-b) + (b-a)(ab+a) - abr, 
        & R_2 &= as + a(b-a) - b(r + p),\\
        L_3 &= (a-1)s^2 + s(ab-a^2-ar+2a) + a^2(r-1), 
        & R_3 &= s(ab+a-r-p) + a(r + p - a).
    \end{align*}
    Then the system of equations given by $L_i - R_i K = 0$ for $i = 1,2,3$ has no integer solutions.
\end{lemma}

\begin{proof}
    Suppose, for the sake of contradiction, that integer solutions exist. We introduce the substitutions $s=a+u$ and $r+p=b+v$.
    The equation $L_1 - R_1 K = 0$, combined with the definition of $K$, implies the identity
    \[
        K = \frac{ar+u}{v+b-a} = 1 + \frac{ar}{v}.
    \]
    Since $K \in \mathbb{Z}$, we define the integer $m = ar/v$, so $K = 1+m$.
    Substituting $K=1+m$ into $L_2 - R_2 K = 0$ simplifies the equation to the constraint $u(m-b) = b^2$, which implies $u \neq 0$. Further simplification of the third equation $L_3 - R_3 K = 0$ gives $u = bm.$
   Substituting this into $u(m-b) = b^2$, we obtain
    \[
        m^2 = b(m+1).
    \]
    It follows that $m+1 \mid m^2$. Since $\gcd(m,m+1)=1$, we must have $m+1 \in \{1,-1\}$.
    If $m+1=1$, then $m=0$, and hence $u=bm=0$, contradicting $u(m-b)=b^2>0$.
    If $m+1=-1$, then $m=-2$, and hence $b=-4$, contradicting $b>0$.
    Thus, no solutions exist.
\end{proof}

\begin{theorem}\label{t2}
The tree $T = T(a_1, \dots, a_r \paramsep b_1, \dots, b_s)$ has exactly three main eigenvalues if and only if $r=s,a_i=b_j=a$, and $T \in T_{r}(a)$.
\end{theorem}

\begin{proof}
Let $T=T(a_1, \dots, a_r \paramsep b_1, \dots, b_s)$, where $r,s\ge 1$, $a_1\le a_2\le \ldots \le a_r$, and $b_1\le b_2\le \ldots \le b_s$. Then $T$ has the equitable $(2r+2s+2)$-partition $\pi$ defined above. By Corollary \ref{1c}, $T$ has exactly three main eigenvalues if and only if  the following linear system has a unique solution:
\begin{equation} \tag{1} \label{sys1}
\left\{
\begin{aligned}
& c_2 + c_1 + (a_i+1)c_0 = a_i + r + 1, && 1 \le i \le r; \\
& c_2 + c_1 + (b_j+1)c_0 = b_j + s + 1, && 1 \le j \le s; \\
& c_2 + (a_i+1)c_1 + (a_i+r+1)c_0 = a_i^2 + a_i + \sigma_a + r + s + 1, && 1 \le i \le r; \\
& c_2 + (b_j+1)c_1 + (b_j+s+1)c_0 = b_j^2 + b_j + \sigma_b + r + s + 1, && 1 \le j \le s; \\
& c_2 + (r+1)c_1 + (\sigma_a + r + s + 1)c_0 = \sigma_a + \sigma_b + (r+1)^2 + s, \\
& c_2 + (s+1)c_1 + (\sigma_b + r + s + 1)c_0 = \sigma_a + \sigma_b + (s+1)^2 + r.
\end{aligned}
\right.
\end{equation}
 where $\sigma_a = \sum_{i=1}^r a_i$ ,
$\sigma_b = \sum_{j=1}^s b_j.$  The uniqueness of the solution can be verified computationally using SageMath (see \hyperref[sec:data-availability]{Data Availability Statement}).

\noindent\textit{Case 1.}
Assume $a_i = a$ for $1 \le i \le r$ and $b_j = b$ for $1 \le j \le s$.
Then $\sigma_a = ra$ and $\sigma_b = sb$, and system \eqref{sys1} is equivalent to
\begin{equation}\tag{2}\label{sys2}
\left\{
\begin{aligned}
& c_2 + c_1 + (a+1)c_0 = a + r + 1, \\[2pt]
& (b-a)c_0 = b - a + s - r, \\[2pt]
& a c_1 + r c_0 = a^2 + a r + s, \\[2pt]
& (a - b)c_1 + (r - s)c_0 = (a - b)(a + b) + a r - b s - (r - s), \\[2pt]
& (r - s)c_1 + (a r - b s)c_0 = (r - s)(r + s + 1),\\[2pt]
& c_2 + (r+1)c_1 + (ar + r + s + 1)c_0 = a r + b s + (r+1)^2 + s.
\end{aligned}
\right.
\end{equation}

\medskip
\noindent\textit{Subcase 1.1.} $a=b$.

From the second equation of \eqref{sys2} we get
\[
0 = (b-a)c_0 = b-a+s-r,
\]
hence $r=s$. Substituting $b=a$ and $s=r$ into \eqref{sys2} and eliminating $c_2$ yields the equivalent system
\begin{equation}\tag{2$'$}\label{sys2'}
\left\{
\begin{aligned}
& c_2 + c_1 + (a+1)c_0 = a + r + 1, \\[2pt]
& a c_1 + r c_0 = a^2 + a r + r, \\[2pt]
& r c_1 + (a r - a + 2r)c_0 = 2 a r + r^2 + 2r - a.
\end{aligned}
\right.
\end{equation}

Let $(c_2,c_1,c_0)^\top$ be the unknown vector and denote by $\Delta$ the determinant of the coefficient matrix of \eqref{sys2'}. A direct calculation gives
$
\Delta =r a^2 + 2ra - a^2 - r^2.
$
Thus system \eqref{sys1} has a unique solution in this subcase if and only if $\Delta \ne 0$.

Assume, for a contradiction, that $(r-1)a^2 = r(r-2a)$ holds for some $a,r \in \mathbb{Z}_{>0}$.
Let $g = \gcd(a,r)$ and write $a = gx$, $r = gy$ with $x,y \in \mathbb{Z}_{>0}$ and $\gcd(x,y)=1$.
Substituting into the equation and dividing by $g^2$ gives
\[
(gy-1)x^2 = y(y-2x).
\]
Hence $y \mid (gy-1)x^2$. Since $\gcd(y,gy-1)=1$, it follows that $y \mid x^2$.
Together with $\gcd(x,y)=1$ this implies $y=1$, and thus
\[
(g-1)x^2 = 1 - 2x.
\]
But for $g \ge 1$ and $x \ge 1$ we have $(g-1)x^2 \ge 0$ and $1-2x \le -1$, a contradiction.
Therefore the equation $(r-1)a^2 = r(r-2a)$ has no solution in positive integers $a,r$.

\noindent{\textit{Subcase 1.2.}} Let $a\ne b$.
In this case, the system \eqref{sys2} is equivalent to

\begin{equation}\tag{3}
\left\{
\begin{aligned}
&  c_2 + c_1 + (a+1)c_0 = a + r + 1, \\[2pt]
& a c_1 + r c_0 = a^2 + s + r a, \\[2pt]
& (b - a)c_0 = b - a + s - r, \\[2pt]
& [ a(b-a+s) - br ] c_0 =  a(b^2+b-a+s+sb) - b(a^2+s+ra)  , \\[2pt]
& [ a(s+ra+r-a) - r^2 ] c_0 =  a(r^2+r+s+ra+sb-a) - r(a^2+s+ra) , \\[2pt]
& [ a(r+sb+s-a) - sr ] c_0 =  a(s^2+2s+ra+sb-a) - s(a^2+s+ra)  .
\end{aligned}
\right.
\end{equation}
By eliminating $c_0$, the linear system \eqref{sys1} admits a unique solution if and only if the system is consistent, which holds if and only if the following Diophantine system has a solution:
\begin{equation} \tag{3$'$}\label{sys3'}
\left\{
\begin{aligned}
& [ a(b^2+b-a+s+sb) - b(a^2+s+ra) ] - [ a(b-a+s) - br ] K = 0, \\[2pt]
& [ a(r^2+r+s+ra+sb-a) - r(a^2+s+ra) ] - [ a(s+ra+r-a) - r^2 ] K = 0, \\[2pt]
& [ a(s^2+2s+ra+sb-a) - s(a^2+s+ra) ] - [ a(r+sb+s-a) - sr ] K = 0.
\end{aligned}
\right.
\end{equation}
where
\[
K = \frac{b - a + s - r}{b - a}.
\]

\par
By Lemma \ref{le2}, the system \eqref{sys3'} has no solution. Therefore, in this case,
the linear system  \eqref{sys1} has no solution.

\noindent{\textit{Case 2.}} Let $t = \min\{i; (a_1 < a_i) \wedge (2 \le i \le r)\}$.

By eliminating $c_2$ , $c_1$ and $c_0,$ the system \eqref{sys1} is equivalent to
\begin{equation}\tag{4}\label{sys4}
\left\{
\begin{aligned}
& a_1c_2 = a_1 r - s - \sigma_a - a_1^2 + r, \\[2pt]
& a_1c_1 = a_1^2 + s + \sigma_a - r, \\[2pt]
& c_0 = 1, \\[2pt]
& s - r= 0, \\[2pt]
&(a_k - a_1) [ a_1 a_k - s - \sigma_a + r ]= 0,&& k=t, \dots, r, \\[2pt]
&a_1 b_j^2 + a_1 \sigma_b - b_j ( a_1^2 + s + \sigma_a - r )= 0, && j=1, \dots, s, \\[2pt]
& r^2(a_1 + 1) - r(s + \sigma_a + a_1^2) + a_1 \sigma_b= 0, \\[2pt]
& s^2(a_1 - 1) + s(a_1 - a_1^2 - \sigma_a + r) + a_1(\sigma_a - r)= 0.
\end{aligned}
\right.
\end{equation}

The linear system \eqref{sys1} admits a unique solution if and only if the system \eqref{sys4} is consistent. Observing that the fourth equation in \eqref{sys4} implies $s=r$, we substitute $s=r$ into the remaining equations. Consequently, the condition holds if and only if the following system has a solution:
\begin{equation}\tag{4$'$}\label{sys4'}
\left\{
\begin{aligned}
& \sigma_a = a_1 a_i,  && i=t, \dots, r, \\[3pt]
& a_1 b_j^2 - b_j(a_1^2 + \sigma_a) + a_1 \sigma_b = 0, && j=1, \dots, s, \\[3pt]
& r a_1 (r - a_1) - (r \sigma_a - a_1 \sigma_b) = 0, \\[3pt]
& (r a_1 - \sigma_a)(r - a_1) = 0.
\end{aligned}
\right.
\end{equation}

 The first equation implies $a_k=a_t$ for $k=t, \dots, r.$ Thus $\sigma_a = (t-1)a_1 + (r-t+1)a_t$ with $a_t > a_1$. The last equation implies $\sigma_a = r a_1$ or $r = a_1$.
  
If $\sigma_a = r a_1$, the first equation gives $r a_1 = a_1 a_t, $ implying $ r = a_t$. Substituting into the expression for $\sigma_a$ yields $r a_1 = (t-1)a_1 + (r-t+1)r$, which simplifies to $(r-t+1)(r-a_1) = 0$. Since $t \le r$, this implies $r=a_1$, contradicting $a_t > a_1$.

If $r = a_1$, the first equation gives $\sigma_a = r a_t$. Substituting into the expression for $\sigma_a$ yields $r a_t = (t-1)r + (r-t+1)a_t$, which simplifies to $(t-1)(r-a_t) = 0$. Since $t \ge 2$, this implies $r=a_t$, again contradicting $a_t > a_1$.
  
Therefore, the system \eqref{sys4'} has no solution. 
By symmetry, the same conclusion holds with the roles of $(a_i,r)$ and $(b_j,s)$ interchanged. In particular, if there exists $t\in\{2,\dots,s\}$ such that $b_t>b_1$, then the linear system \eqref{sys1} cannot have a unique solution.
\end{proof}

\begin{theorem}\label{t3}
Let $p,q \ge 1$. The tree $T = T^{p,q}(a_1, \ldots, a_r \paramsep b_1, \ldots, b_s)$ has exactly three main eigenvalues if and only if $T \in \mathcal{T}$.
\end{theorem}
\begin{proof}
Let $T = T^{p,q}(a_1, \ldots, a_r \paramsep b_1, \ldots, b_s)$. Then $T$ has an equitable $(2r+2s+4)$-partition $\pi$ defined above. By Corollary \ref{1c}, $T$ has exactly three main eigenvalues if and only if the linear system \eqref{sys5} below has a unique solution:
{\small
\begin{equation}
\tag{5}\label{sys5}
\makebox[0pt][l]{\hspace*{-.5\linewidth}$\displaystyle
\left\{
\begin{aligned}
    & \begin{aligned}
        & c_2 + c_1 + (a_i+1)c_0 = a_i+r+p+1, && 1 \le i \le r; \\[2pt]
        & c_2 + c_1 + (b_j+1)c_0 = b_j+s+q+1, && 1 \le j \le s; \\[2pt]
        & c_2 + c_1 + (r+p+1)c_0 = \sigma_a+r+s+p+q+1, \\[2pt]
        & c_2 + c_1 + (s+q+1)c_0 = \sigma_b+r+s+p+q+1, \\[2pt]
        & c_2 + (a_i+1)c_1 + (a_i+r+p+1)c_0 = a_i^2+a_i+\sigma_a+r+s+p+q+1, && 1 \le i \le r; \\[2pt]
        & c_2 + (b_j+1)c_1 + (b_j+s+q+1)c_0 = b_j^2+b_j+\sigma_b+r+s+p+q+1, && 1 \le j \le s;
    \end{aligned} \\[2pt]
    & c_2 + (r+p+1)c_1 + (\sigma_a+r+s+p+q+1)c_0 = \sigma_a+\sigma_b+(r+p)(r+p+1)+r+s+p+q+1, \\[2pt]
    & c_2 + (s+q+1)c_1 + (\sigma_b+r+s+p+q+1)c_0 = \sigma_a+\sigma_b+(s+q)(s+q+1)+r+s+p+q+1.
\end{aligned}
\right.
$}
\end{equation}
}
 where $\sigma_a = \sum_{i=1}^r a_i$ ,
$\sigma_b = \sum_{j=1}^s b_j.$

\noindent{\textit{Case 1.}} Let $a_i=a  $ for $ 1\le i\le r $  and $b_j=b  $ for $ 1\le j\le s .$ Then $\sigma_a = ra$ and $\sigma_b = sb$, and system \eqref{sys5} is equivalent to
\begin{equation}\tag{6}\label{sys6}
\makebox[0pt][l]{\hspace*{-.5\linewidth}$\displaystyle
\left\{
\begin{aligned}
& a\,c_2+\bigl(a^2+a-r-p\bigr)c_0 \;=\; a(p+1)-(s+q),\\[2pt]
& a\,c_1+(r+p)c_0 \;=\; a^2+ar+(s+q),\\[2pt]
& (b-a)c_0 \;=\; b-a+s+q-r-p,\\[2pt]
& (r+p-a)c_0 \;=\; a(r-1)+(s+q),\\[2pt]
& (s+q-a)c_0 \;=\; s(b+1)+q-a,\\[2pt]
& \bigl(ab-a^2+a(s+q)-b(r+p)\bigr)c_0
\;=\;ab\bigl(b+s-r+1\bigr)-a^2(b+1)+(a-b)(s+q),\\[2pt]
& \bigl(a^2(r-1)+a(r+p+s+q)-(r+p)^2\bigr)c_0
\;=\;abs+(a-r-p)\bigl[s+q-a(p+1)\bigr],\\[2pt]
& \bigl(abs-a^2+a(r+p)+(a-r-p)(s+q)\bigr)c_0
\;=\;abs+(a-s-q)\bigl[ar-(a-1)(s+q)\bigr]-a^2+a(s+q).
\end{aligned}
\right.
$}
\end{equation}

\noindent{\textit{Subcase 1.1.}} $a=b.$

Then the third equation of \eqref{sys6} gives $s+q=r+p$.
Substituting this into the sixth equation yields $a^2(s-r)=0$. Since $a>0$, we have $s=r$, hence $p=q$.
Substituting $b=a$ and $s=r$ into \eqref{sys6} yields the equivalent system
\begin{equation}\tag{6$'$}\label{sys6'}
\left\{
\begin{aligned}
& a\,c_2+\bigl(a^2+a-r-p\bigr)c_0 \;=\; ap+a-r-p,\\[2pt]
& a\,c_1+(r+p)c_0 \;=\; a^2+ar+r+p,\\[2pt]
& (r+p-a)c_0 \;=\; ar+r+p-a,\\[2pt]
& \bigl(a^2(r-1)+2a(r+p)-(r+p)^2\bigr)c_0
\;=\; a^2r+(a-r-p)\bigl[r+p-a(p+1)\bigr].\\
\end{aligned}
\right.
\end{equation}

From the third equation of \eqref{sys6'} we have $r+p\neq a$; otherwise it would force $ar=0$. Let $\rho=r+p$. Eliminating $c_0$ from the last equations of \eqref{sys6'} yields the equivalent reduced system
\begin{equation}\tag{6$''$}\label{sys6''}
\left\{
\begin{aligned}
& a\,c_2+\bigl(a^2+a-\rho\bigr)c_0 \;=\; ap+a-\rho,\\[2pt]
& a\,c_1+\rho\,c_0 \;=\; a^2+ar+\rho,\\[2pt]
& (\rho-a)c_0 \;=\; ar+\rho-a,\\[2pt]
& L_1-L_2 \;=\;0.
\end{aligned}
\right.
\end{equation}
\begin{align*}
\text{where }\quad 
L_1&=(\rho-a)\Bigl[a^2r+(a-\rho)\bigl(\rho-a(p+1)\bigr)\Bigr],\\
L_2&=\Bigl[a^2(r-1)+2a\rho-\rho^2\Bigr]\bigl(ar+\rho-a\bigr).
\end{align*}

System \eqref{sys5} admits a unique solution if and only if system \eqref{sys6''} is consistent, i.e., $L_1 = L_2$.
Let $X := \rho-a$ and $N := ra^2 - X^2$. Then $L_1$ and $L_2$ become
\[
    L_1 = N + apX, \qquad L_2 = \frac{N}{X}(ra+X).
\]
Equating $L_1 = L_2$ yields $apX^2 = raN.$ Substituting $N = ra^2 - X^2$ yields $pX^2 = r(ra^2 - X^2)$, which simplifies to $(r+p)X^2 = (ra)^2$. Since $\rho=r+p$, we obtain the Diophantine equation
\begin{equation}\label{eq:k_sq}
    \rho(\rho-a)^2 = (ra)^2.
\end{equation}
Equation \eqref{eq:k_sq} implies that $\rho$ is a perfect square. Let $\rho = w^2$ for some integer $w > 0$. Extracting the square root gives $w |w^2 - a| = ra$. Since $ra > 0$, we have $w^2 \neq a$; thus there exists $\delta\in\{1,-1\}$ such that
\[
|w^2-a|=\delta(w^2-a),
\]
which implies
\[
 w^3=a(w+\delta r).
\]
Combined with $r+p=w^2$, the uniqueness condition is equivalent to the existence of $\delta\in\{1,-1\}$ such that
\[
\begin{cases}
r+p=w^2,\\
a(w+\delta r)=w^3,
\end{cases}
\]
where necessarily $w+\delta r>0$. The solutions to this system are characterized in Lemma~\ref{le1}.

\noindent{\textit{Subcase 1.2.}} Assume that $a\neq b$, and set
\[
K:=1+\frac{s+q-r-p}{\,b-a\,}.
\]
Using the third equation of \eqref{sys6}, eliminating $c_0$ from the fourth to the last equations yields the equivalent system
{\small
\begin{equation}\tag{7}\label{sys7}
\makebox[0pt][l]{\hspace*{-.5\linewidth}$\displaystyle
\left\{
\begin{aligned}
& a\,c_2+\bigl(a^2+a-r-p\bigr)c_0 \;=\; a(p+1)-(s+q),\\[2pt]
& a\,c_1+(r+p)c_0 \;=\; a^2+ar+(s+q),\\[2pt]
& (b-a)c_0 \;=\; b-a+s+q-r-p,\\[2pt]
& \bigl[ra-a+s+q\bigr]-\bigl(r+p-a\bigr)K \;=\; 0,\\[2pt]
& \bigl[sb+s+q-a\bigr]-\bigl(s+q-a\bigr)K \;=\; 0,\\[2pt]
& \Bigl[ab\bigl(b+s-r+1\bigr)-a^2(b+1)+(a-b)(s+q)\Bigr]
-\bigl[ab-a^2+a(s+q)-b(r+p)\bigr]K \;=\; 0,\\[2pt]
& \Bigl[abs+(a-r-p)\bigl(s+q-a(p+1)\bigr)\Bigr]
-\bigl[a^2(r-1)+a(r+p+s+q)-(r+p)^2\bigr]K \;=\; 0,\\[2pt]
& \Bigl[abs+(a-s-q)\bigl[ar-(a-1)(s+q)\bigr]-a^2+a(s+q)\Bigr]
-\bigl[abs-a^2+a(r+p)+(a-r-p)(s+q)\bigr]K \;=\; 0.
\end{aligned}
\right.
$}
\end{equation}
}
In this case, the linear system \eqref{sys5} admits a unique solution if and only if the reduced system \eqref{sys7} is consistent. However, by Lemma~\ref{le3}, the subsystem formed by the fourth, fifth, and sixth equations in \eqref{sys7} has no solution. Consequently, system \eqref{sys7} is inconsistent, implying that the linear system \eqref{sys5} has no solution.

\noindent{\textit{Case 2.}} Let $t = \min\{i; (a_1 < a_i) \wedge (2 \le i \le r)\}$.

By eliminating $c_2,c_1$ and $c_0$ the system \eqref{sys5} is equivalent to
\begin{equation}\tag{8}\label{sys8}
\makebox[0pt][l]{\hspace*{-.5\linewidth}$\displaystyle
\left\{
\begin{aligned}
& a_1c_2+\bigl(a_1^2+a_1-r-p\bigr)c_0
   \;=\; a_1(r+p+1)-\sigma_a-s-q, \\[2pt]
& a_1c_1+(r+p)c_0
   \;=\; a_1^2+\sigma_a+s+q, \\[2pt]
& c_0 \;=\; 1,  \\[2pt]
& s+q-r-p \;=\; 0, \\[2pt]
& \sigma_a+s+q-r-p \;=\; 0, \\[2pt]
& \sigma_b \;=\; 0, \\[2pt]
& a_1a_i-\sigma_a-s-q+r+p \;=\; 0, && t\le i\le r, \\[2pt]
& a_1b_j^2+a_1\sigma_b-b_j\bigl(a_1^2+\sigma_a+s+q-r-p\bigr) \;=\; 0,
&& 1\le j\le s, \\[2pt]
& (r+p)^2(a_1+1)+a_1\sigma_b-(r+p)\bigl(a_1^2+\sigma_a+s+q\bigr)\;=\;0, \\[2pt]
& (a_1-1)(s+q)^2-a_1^2(s+q)+(a_1-s-q)\sigma_a+(r+p)(s+q-a_1)+a_1(s+q)\;=\;0.
\end{aligned}
\right.
$}
\end{equation}
However, the sixth equation of the system \eqref{sys8} forces $\sigma_b = 0$, which contradicts $b_j>0$. Hence, in this case the system \eqref{sys5} admits no solution.

By symmetry, if there exists an index $t \in \{2, \dots, s\}$ such that $b_t > b_1$, then the linear system \eqref{sys5} does not admit a unique solution either.
\end{proof}

\begin{theorem}\label{t4}
For $p \ge 1$, the tree $T = T^{p,0}(a_1, \ldots, a_r \paramsep b_1, \ldots, b_s)$ does not have exactly three main eigenvalues.
\end{theorem}

\begin{proof}
Let $T = T^{p,0}(a_1, \ldots, a_r \paramsep b_1, \ldots, b_s)$. Then $T$ has the equitable $(2r+2s+3)$-partition $\pi$ defined above. By Corollary \ref{1c}, $T$ has exactly three main eigenvalues if and only if the linear system \eqref{sys9} below has a unique solution:
{\small
\begin{equation} \tag{9}\label{sys9}
\makebox[0pt][l]{\hspace*{-.5\linewidth}$\displaystyle
\left\{
\begin{aligned}
& c_2 + c_1 + (a_i + 1)c_0 = a_i + r + p + 1, && \hspace{-5em}1 \le i \le r; \\
& c_2 + c_1 + (b_j + 1)c_0 = b_j + s + 1, && \hspace{-5em}1 \le j \le s; \\
& c_2 + c_1 + (r + p + 1)c_0 = \sigma_a + r + s + p + 1, \\
& c_2 + (a_i + 1)c_1 + (a_i + r + p + 1)c_0 = a_i^2 + a_i + \sigma_a + r + s + p + 1, && \hspace{-5em}1 \le i \le r; \\
& c_2 + (b_j + 1)c_1 + (b_j + s + 1)c_0 = b_j^2 + b_j + \sigma_b + r + s + p + 1, && \hspace{-5em}1 \le j \le s; \\
& c_2 + (r + p + 1)c_1 + (\sigma_a + r + s + p + 1)c_0 = \sigma_a + \sigma_b + (r + p)(r + p + 1) + r + s + p + 1, \\
& c_2 + (s + 1)c_1 + (\sigma_b + r + s + p + 1)c_0 = \sigma_a + \sigma_b + s(s + 1) + r + s + p + 1.
\end{aligned}
\right.
$}
\end{equation}
}

\text{where } $\sigma_a = \sum_{i=1}^r a_i,$ $  \sigma_b = \sum_{j=1}^s b_j.$

\noindent\textit{Case 1.}
Assume $a_i = a$ for $1 \le i \le r$ and $b_j = b$ for $1 \le j \le s$.
Then $\sigma_a = ra$ and $\sigma_b = sb$, and system \eqref{sys9} is equivalent to
\begin{equation}\tag{9$'$}\label{sys9'}
\left\{
\begin{aligned}
& c_2+c_1+(a+1)c_0 = a+r+p+1,\\[2pt]
& (b-a)c_0 = b-a+s-r-p,\\[2pt]
& (r+p-a)c_0 = a(r-1)+s,\\[2pt]
& a\,c_1+(r+p)c_0 = a^2+ra+s,\\[2pt]
& b\,c_1+(b-a+s)c_0 = b^2+(s+1)b+s-a,\\[2pt]
& (r+p)c_1+\bigl(ra+r+s+p-a\bigr)c_0
   = ra+sb+(r+p)(r+p+1)+s-a,\\[2pt]
& s\,c_1+\bigl(sb+r+s+p-a\bigr)c_0
   = ra+sb+s(s+1)+s-a.
\end{aligned}
\right.
\end{equation}

\noindent{\textit{Subcase 1.1.}} $a=b$

The second equation of \eqref{sys9'} implies $(b-a)c_0 = s - r - p$. Since $a=b$, this forces $s - r - p = 0$, and hence $s = r+p$. Substituting $b=a$ and $s=r+p$ into the fourth and fifth equations of  \eqref{sys9'} and eliminating $c_1$ yields
\[
ap = 0.
\]
Since $a \ge 1$ and $p \ge 1$, this leads to a contradiction. Therefore, the system admits no solution in this subcase.

\noindent{\textit{Subcase 1.2.}} $a\neq b$.

Let $K$ denote the value of $c_0$ uniquely determined by the second equation of \eqref{sys9'}, that is,
\[
K := 1+\frac{s-r-p}{\,b-a\,}.
\]
Applying Gaussian elimination to \eqref{sys9'} yields the equivalent system
\begin{equation}\tag{10}\label{sys10}
\makebox[0pt][l]{\hspace*{-.5\linewidth}$\displaystyle
\left\{
\begin{aligned}
& c_2+c_1+(a+1)c_0 = a+r+p+1,\\[2pt]
& a\,c_1+(r+p)c_0 = a^2+ra+s,\\[2pt]
& (b-a)c_0 = b-a+s-r-p,\\[2pt]
& (ar+s-a)-(r+p-a)K = 0,\\[2pt]
& \Bigl[\, s(ab+a-b) + (b-a)(ab+a) - abr \,\Bigr]
  - \Bigl[\, as + a(b-a) - b(r+p) \,\Bigr]K = 0,\\[2pt]
& \Bigl[\, s(ab+a-r-p) + a(p+1)(r+p-a) \,\Bigr]
  - \Bigl[\, as + a^2(r-1) - (r+p)(r+p-a) \,\Bigr]K = 0,\\[2pt]
& \Bigl[\, (a-1)s^2 + s(ab-a^2-ar+2a) + a^2(r-1) \,\Bigr]
  - \Bigl[\, s(ab+a-r-p) + a(r+p-a) \,\Bigr]K = 0.
\end{aligned}
\right.
$}
\end{equation}

In this case, the linear system \eqref{sys9} admits a solution if and only if the reduced system \eqref{sys10} is consistent. However, the subsystem composed of the fourth, fifth, and seventh equations in \eqref{sys10}  corresponds precisely to the conditions analyzed in Lemma~\ref{lem:no-solution-3eq}. By Lemma~\ref{lem:no-solution-3eq}, this subsystem admits no  solution. Consequently, system \eqref{sys10}  is inconsistent, implying that the linear system \eqref{sys9} has no solution.

\noindent{\textit{Case 2.}} Let $t = \min\{i; (a_1 < a_i) \wedge (2 \le i \le r)\}$.

Eliminating $c_2, c_1, c_0$ from \eqref{sys9} yields the equivalent system
\begin{equation}\tag{11}
\left\{
\begin{aligned}
& c_2 + c_1 + (a_1+1)c_0 = a_1 + r + p + 1, \\[2pt]
& a_1 c_1 + (r+p)c_0 = a_1^2 + s + \sigma_a, \\[2pt]
& c_0 = 1, \\[2pt]
& s - r - p = 0, \\[2pt]
& s + \sigma_a - r - p = 0, \\[2pt]
& a_1 a_i + r + p - s - \sigma_a = 0, && t\le i\le r, \\[2pt]
& a_1 \sigma_b + b_j\bigl(a_1 b_j - a_1^2 + r + p - s - \sigma_a\bigr) = 0,\qquad 1\le j\le s, \\[2pt]
& (a_1+1)(r+p)^2 - (a_1^2 + s + \sigma_a)(r+p) + a_1\sigma_b = 0, \\[2pt]
& (s-a_1)\bigl(a_1 s + r + p - s - \sigma_a\bigr) = 0.
\end{aligned}
\right.
\end{equation}
Subtracting the fourth equation from the fifth immediately implies $\sigma_a = 0$, which contradicts the fact that $\sigma_a \ge r \ge 1$.
Analogously, assuming there exists $t \in \{2, \dots, s\}$ such that $b_t > b_1$ leads to a contradiction.
Consequently, the system is inconsistent, and $T^{p,0}(a_1, \ldots, a_r \paramsep b_1, \ldots, b_s)$ cannot admit exactly three main eigenvalues.
\end{proof}

\begin{corollary}
Let $T = T^{p,q}(a_1, \dots, a_r \paramsep b_1, \dots, b_s)$ be the tree of diameter 5 with exactly three main eigenvalues. Then:
$$
\varphi_M(T, x) = 
\begin{cases} 
x^3 - (1+\delta \kappa)x^2 - (a-\delta \kappa)x - (\kappa^2 - \delta a\kappa - a), & \text{if } T \in \mathcal{T}; \\
x^3 - x^2 - (a+r)x + a, & \text{if } T \cong T_r(a).
\end{cases}
$$

\end{corollary}

\begin{proof}
If $T \cong T_r(a)$, consider the equitable tripartition $\pi(T_r(a))$. The divisor matrix is
\[
  B_\pi =
  \begin{pmatrix}
    0 & 1 & 0\\
    a & 0 & 1\\
    0 & r & 1
  \end{pmatrix},
\]
hence
\[
  \varphi_M(T,x)
  = \det(xI - B_\pi)
  = x^3 - x^2 - (a+r)x + a.
\]

If $T \in \mathcal{T}$, then by Corollary \ref{1c} and the linear systems obtained earlier, the unique coefficients $c_0, c_1, c_2$ in
\[
  d_3(v) = c_0 d_2(v) + c_1 d(v) + c_2
\]
are given by
\[
  (c_0, c_1, c_2) = (1+\delta \kappa, \ a - \delta \kappa, \ \kappa^2 - \delta a\kappa - a),
\]
where  $a = \kappa^3/\eta$ as before. Consequently,
\begin{align*}
  \varphi_M(T,x)
  &= x^3 - c_0 x^2 - c_1 x - c_2 \\
  &= x^3 - (1+\delta \kappa)x^2 - (a - \delta \kappa)x - (\kappa^2 - \delta a\kappa - a).
\end{align*}
\end{proof}

\section{Conclusion}

In this paper, we have provided a complete classification of trees of diameter $5$ with exactly three main eigenvalues. By leveraging the theory of equitable partitions, we reduced the spectral condition to explicit Diophantine constraints on the structural parameters. The unique solvability, and in particular the consistency, of the resulting systems were rigorously determined by using Gr\"obner basis theory to handle the intricate algebraic constraints. This work extends the classification of trees of diameter $4$ established by Fran\c{c}a et al.~\cite{FrBrJa}.

The following example shows that, unlike the case of two main eigenvalues \cite{HouZhou2005}, no diameter bound holds for trees with exactly three main eigenvalues.

We recall the notation of generalized Bethe trees in the sense of Rojo and Soto \cite{RoSo}. Let $B(d_h,d_{h-1},\ldots,d_2)$ be the rooted tree with levels $C_h,C_{h-1},\ldots,C_2,C_1$, where $C_h$ consists of the root and $C_1$ consists of the ordinary leaves. Each vertex in $C_h$ has $d_h$ children in $C_{h-1}$; for $2\le i\le h-1$, each vertex in $C_i$ has one parent in $C_{i+1}$ and $d_i-1$ children in $C_{i-1}$; and the vertices in $C_1$ have degree $1$. For nonnegative integers $t_h,t_{h-1},\ldots,t_2$, we call $B(d_h,d_{h-1},\ldots,d_2 \paramsep t_h,t_{h-1},\ldots,t_2)$ the level-uniform pendent extension obtained from $B(d_h,d_{h-1},\ldots,d_2)$ by attaching $t_i$ pendent vertices to each vertex of $C_i$, for $i=2,\ldots,h$.

\begin{example}\label{ex:main-unbounded-diameter}
For every integer $h\ge 6$, let
\[
T_h =
B(\underbrace{7,\ldots,7}_{h-5},11,3,163,13
\paramsep
\underbrace{84,\ldots,84}_{h-5},80,52,0,0).
\]
Equivalently, for the levels $C_h,\ldots,C_1$ we have
\[
(d_i,t_i)=
\begin{cases}
(7,84),&6\le i\le h,\\
(11,80),&i=5,\\
(3,52),&i=4,\\
(163,0),&i=3,\\
(13,0),&i=2.
\end{cases}
\]
Let the cells be ordered as
\[
(C_h,C_{h-1},\ldots,C_1,P_h,P_{h-1},\ldots,P_4),
\]
where $P_i$ is the set of pendent vertices attached to $C_i$. The divisor matrix $B_h=(b_{X,Y})$ is determined by
\[
b_{X,Y}=
\begin{cases}
7, & (X,Y)=(C_h,C_{h-1}),\\
84, & (X,Y)=(C_h,P_h),\\
1, & (X,Y)=(C_i,C_{i+1}),\ 4\le i\le h-1,\\
d_i-1, & (X,Y)=(C_i,C_{i-1}),\ 4\le i\le h-1,\\
t_i, & (X,Y)=(C_i,P_i),\ 4\le i\le h-1,\\
1, & (X,Y)=(C_3,C_4),(C_2,C_3),(C_1,C_2),\\
162, & (X,Y)=(C_3,C_2),\\
12, & (X,Y)=(C_2,C_1),\\
1, & (X,Y)=(P_i,C_i),\ 4\le i\le h,\\
0, & \text{otherwise}.
\end{cases}
\]
For $v_i=B_h^i\mathbf{j}$, where $\mathbf{j}$ has length $2h-3$, we have
\[
\begin{aligned}
v_0&=\mathbf{j},\\
v_1&=(\underbrace{91,\ldots,91}_{h-4},55,163,13,1
\paramsep \underbrace{1,\ldots,1}_{h-3})^{\top},\\
v_2&=(\underbrace{721,\ldots,721}_{h-4},469,2161,175,13
\paramsep \underbrace{91,\ldots,91}_{h-4},55)^{\top},\\
v_3&=(\underbrace{12691,\ldots,12691}_{h-4},7903,28819,2317,175
\paramsep \underbrace{721,\ldots,721}_{h-4},469)^{\top}
      =7v_2+84v_1.
\end{aligned}
\]
Thus $\rank W(B_h)=3$. By Proposition~\ref{prop:quotient-walk-rank}, $T_h$ has exactly three main eigenvalues. Hence, by Theorem~\ref{t1},
\[
  \varphi_M(T_h,x)=x^3-7x^2-84x=x(x^2-7x-84).
\]
Also $\operatorname{diam}(T_h)=2(h-1)$. Hence trees with exactly three main eigenvalues have unbounded diameter.
\end{example}

\begin{figure}[!ht]
\centering
\begin{tikzpicture}[
  x=1cm,y=1cm,
  v/.style={circle, fill=black, inner sep=1.35pt},
  e/.style={draw, line width=0.42pt, line cap=round},
  lab/.style={font=\scriptsize, inner sep=0.8pt},
  br/.style={decorate, decoration={brace, mirror, amplitude=3pt}, line width=0.35pt}
]
  \node[v] (Ch) at (0,0) {};
  \node[v] (Chm) at (1.25,0) {};
  \node[v] (Csix) at (2.85,0) {};
  \node[v] (Cfive) at (4.20,0) {};
  \node[v] (Cfour) at (5.55,0) {};
  \node[v] (Cthree) at (6.85,0) {};
  \node[v] (Ctwo) at (8.15,0) {};
  \node[lab, below=3pt] at (Ch) {$C_h$};
  \node[lab, below=3pt] at (Chm) {$C_{h-1}$};
  \node[lab, below=3pt] at (Csix) {$C_6$};
  \node[lab, below=3pt] at (Cfive) {$C_5$};
  \node[lab, below=3pt] at (Cfour) {$C_4$};
  \node[lab, below=3pt] at (Cthree) {$C_3$};
  \node[lab, below=3pt] at (Ctwo) {$C_2$};
  \node[lab, below=18pt] at (Ch) {$(7,84)$};
  \node[lab, below=18pt] at (Chm) {$(7,84)$};
  \node[lab, below=18pt] at (Csix) {$(7,84)$};
  \node[lab, below=18pt] at (Cfive) {$(11,80)$};
  \node[lab, below=18pt] at (Cfour) {$(3,52)$};
  \node[lab, below=18pt] at (Cthree) {$(163,0)$};
  \node[lab, below=18pt] at (Ctwo) {$(13,0)$};
  \foreach \name/\x/\m in {Ch/0/84,Chm/1.25/84,Csix/2.85/84,Cfive/4.20/80,Cfour/5.55/52}{
    \node[v] (\name-p1) at (\x-0.16,0.82) {};
    \node[v] (\name-p2) at (\x+0.16,1.42) {};
    \draw[e] (\name)--(\name-p1);
    \draw[e] (\name)--(\name-p2);
    \node[lab] at (\x,1.10) {$\vdots$};
    \node[lab, above=2pt] at (\x,1.50) {$\m$};
  }
  \draw[e] (Ch)--(Chm);
  \draw[e,densely dotted] (Chm)--(Csix);
  \draw[e] (Csix)--(Cfive)--(Cfour)--(Cthree)--(Ctwo);
  \node[lab, left=5pt] at (Ch) {root};

  \foreach \k/\dy in {1/0.36,2/0,3/-0.36}{
    \node[v] (leaf\k) at (9.25,\dy) {};
    \draw[e] (Ctwo)--(leaf\k);
  }
  \node[lab] at (9.25,0.18) {$\vdots$};
  \node[lab, right=3pt] at (leaf2) {$12$ leaves};

  \draw[br] (-0.18,-1.42)--(3.03,-1.42) node[lab, midway, below=4pt] {$h-5$ levels};
\end{tikzpicture}
\caption{Labeled $T_h$.}
\label{fig:infinite_family_Th}
\end{figure}
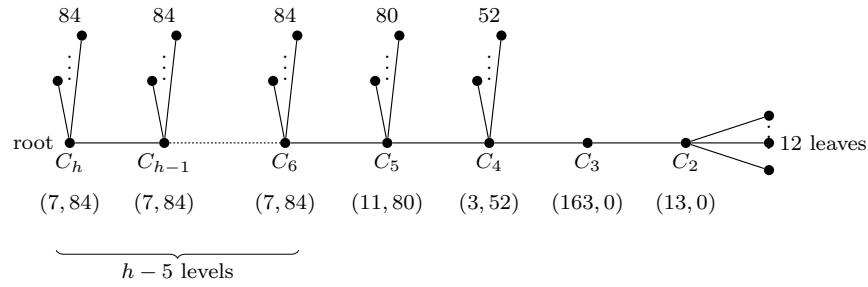

This family shows that the diameters of trees with exactly three main eigenvalues are unbounded.

\section*{Data Availability Statement}
\phantomsection
\label{sec:data-availability}
The verification computations and exhaustive search programs for trees with exactly three main eigenvalues in this paper were implemented in SageMath and are
publicly available as Jupyter notebook files. The corresponding repository is hosted on GitHub at
\url{https://github.com/ditf015/trees-with-3-main-eigenvalues}.
A citable archived release (v1.1.1) is available on Zenodo:
\url{https://doi.org/10.5281/zenodo.19440113}.

\bibliographystyle{abbrv}
\bibliography{refs}

\end{document}